\documentclass[12pt]{amsart}
\usepackage{amssymb,bbold}
\usepackage[all]{xy}

\theoremstyle{plain}
\newtheorem{theorem}{Theorem}
\newtheorem{corollary}[theorem]{Corollary}
\newtheorem{lemma}[theorem]{Lemma}
\newtheorem{proposition}[theorem]{Proposition}

\theoremstyle{definition}

\newtheorem{remark}[theorem]{Remark}

\renewcommand{\ge}{\geqslant}

\begin{document}
\baselineskip 18pt

\title[Bounded orthomorphisms]
      {Bounded orthomorphisms between locally solid vector lattices}

\author[R. ~Sabbagh]{Raheleh Sabbagh}
\author[O.~Zabeti]{Omid Zabeti}


\address[R. ~Sabbagh and O.~Zabeti]
  {Department of Mathematics, Faculty of Mathematics, Statistics, and Computer science,
   University of Sistan and Baluchestan, Zahedan,
   P.O. Box 98135-674. Iran}
\email{o.zabeti@gmail.com}
\email{sabaghrahele@gmail.com}

\begin{abstract}
The main aim of the present note is to consider bounded orthomorphisms between locally solid vector lattices. We establish a version of the remarkable Zannen theorem regarding equivalence between orthomomorphisms and the underlying vector lattice to the case of all bounded orthomomorphisms. Furthermore, we investigate topological and ordered structures for these classes of orthomorphisms, as well.
\end{abstract}

\date{\today}

\maketitle
\section{motivation and introduction}
Let us start with some motivation. Suppose $X$ is an Archimedean vector lattice and $Orth(X)$ is the space of all orthomorphisms on $X$. This space has many important consequences using just the order structure ( see \cite[Section 2.3]{AB}). One of the most remarkable advantages in $Orth(X)$ is the pointwise lattice operations formulae for calculating the suprema or infima of these operators. On the other hand, there are several non-equivalent ways to define bounded operators between locally solid vector lattices; furthermore, these spaces have some ordered and topological structures, as well ( see \cite{EGZ, Z1} for a detailed exposition).
Therefore, it is natural to expect some special properties from bounded orthomorphisms defined on a locally solid vector lattice. This is what our paper is about. We shall consider some topological and ordered structures for different types of bounded orthomorphisms. In particular, we establish a version of the known Zannen theorem ( \cite[Theorem 2.62]{AB}) to each category of bounded orthomorphisms. Moreover, we investigate topologically and ordered closedness for these classes of operators, as well. Now, let us recall some preliminaries we need in the sequel.

A vector lattice $X$ is called {\em order complete} if every non-empty bounded above subset of $X$ has a supremum and $X$ is {\em Archimedean} if $nx\leq y$ for each $n\in \Bbb N$ implies that $x\leq 0$. It is known that every order complete vector lattice is Archimedean. A set $S\subseteq X$ is called a {\em solid} set if $x\in X$, $y\in S$ and $|x|\leq |y|$ imply that $x\in S$. Also, recall that a linear topology $\tau$ on a vector lattice $X$ is referred to as  {\em locally solid} if it has a local basis at zero consisting of solid sets.

Suppose $X$ is a locally solid vector lattice. A net $(x_{\alpha})\subseteq X$ is said to be {\em order} convergent to $x\in X$ if there exists a net $(z_{\beta})$ ( possibly over a different index set) such that $z_{\beta}\downarrow 0$ and for every $\beta$, there is an $\alpha_0$ with $|x_{\alpha}-x|\leq z_{\beta}$ for each $\alpha\ge \alpha_0$. A set $A\subseteq X$ is called {\em order closed} if it contains limits of all order convergent nets which lie in $A$.
 Keep in mind that a topology $\tau$ on a vector lattice $X$ making it a locally solid vector lattice is referred to as {\em Fatou} if it has a local basis at zero consisting of solid order closed neighborhoods. In this case, we say that $X$ has the {\em Fatou property}.
 Observe that a locally solid vector lattice $(X,\tau)$ is said to have the {\em Levi property} if every $\tau$-bounded upward directed set in $X_{+}$ has a supremum.
 Finally, recall that a locally solid vector lattice $(X,\tau)$ possesses the {\em Lebesgue property} if for every net $(u_{\alpha})$ in $X$, $u_{\alpha}\downarrow 0$ implies that $u_{\alpha}\xrightarrow{\tau}0$.
 
Recall that, for an Archimedean vector lattice $X$, by $Orth(X)$, we mean the space of all orthomorphisms on $X$; more precisely, an order bounded band preserving operator on $X$ is called an orthomorphism. For more details on this subject, see \cite{AB}. Observe that a linear operator $T$ on a vector lattice $X$ is called band preserving if $x\bot y$ in $X$ implies that $T(x)\bot y$. Note that by $x\bot y$, we mean $|x|\wedge |y|=0$. Now, suppose $X$ is a locally solid vector lattice. An orthomorphism $T$ on $X$ is called $nb$-bounded if there exists a zero neighborhood $U\subseteq X$ such that $T(U)$ is bounded in $X$; $T$ is said to be $bb$-bounded provided that it maps bounded sets into bounded sets. 
First, observe that compatible with the different spaces of all bounded operators on $X$, these notions of bounded orthomorphisms are not equivalent, in general. Consider the identity operator on ${\Bbb R}^{\Bbb N}$; it is $bb$-bounded and continuous but not $nb$-bounded. Suppose $X$ is $c_{00}$ with the norm topology. Consider the orthomorphism $T$ on $X$ defined via $T((x_n))=(nx_n)$; indeed, it is neither $bb$-bounded nor continuous.

The class of all $nb$-bounded orthomorphisms on $X$  is denoted by $Orth_{n}(X)$ and is equipped with the topology of uniform convergence on some zero neighborhood, namely, a net $(S_{\alpha})$ of $nb$-bounded orthomorphisms converges to zero on some zero neighborhood $U\subseteq X$ if for any zero neighborhood $V\subseteq X$ there is an $\alpha_0$ such that $S_{\alpha}(U) \subseteq V$ for each $\alpha\geq\alpha_0$. The class of all $bb$-bounded orthomorphisms on $X$ is denoted by $Orth_{b}(X)$ and is allocated to the topology of uniform convergence on bounded sets. Recall that a net $(S_{\alpha})$ of $bb$-bounded orthomorphisms uniformly converges to zero on a bounded set $B\subseteq X$ if for any zero neighborhood $V \subseteq X$ there is an $\alpha_0$ with $S_{\alpha}(B) \subseteq V$ for each $\alpha\geq\alpha_0$.

The class of all continuous orthomorphisms on $X$ is denoted by $Orth_{c}(X)$ and is equipped with the topology of equicontinuous convergence, namely, a net $(S_{\alpha})$ of continuous orthomorphisms converges equicontinuously to zero if for each zero neighborhood $V\subseteq X$ there is a zero neighborhood $U\subseteq X$ such that for every $\varepsilon>0$ there exists an $\alpha_0$ with $S_{\alpha}(U)\subseteq \varepsilon V$ for each $\alpha\geq\alpha_0$. See \cite{Tr} for a detailed exposition on these classes of operators. In general, we have $Orth_n(X)\subseteq Orth_c(X)\subseteq Orth_b(X)$ and when $X$ is locally bounded, they coincide.

 Furthermore, suppose $X$ is a locally solid vector lattice. We say that $X$ has the {\em $AM$-property} provided that for every bounded set $B\subseteq X$, $B^{\vee}$ is also bounded with the same scalars; namely, given a zero neighborhood $V$ and any positive scalar $\alpha$ with $B\subseteq \alpha V$, we have  $B^{\vee}\subseteq \alpha V$. Observe that by $B^{\vee}$, we mean the set of all finite suprema of elements of $B$; for ample information, see \cite{Z1}.

 All vector lattices in this note are assumed to be Archimedean. For undefined terminology and related topics, see \cite{AB1,AB}.
\section{main result}
First, we have the following useful facts.
\begin{lemma}\label{10}
Suppose $X$ is a locally solid vector lattice. Then $Orth_{n}(X)$, $Orth_{b}(X)$, and $Orth_{c}(X)$ are vector lattices.
\end{lemma}
\begin{proof}
It is enough to prove that, in each case, the modulus of the orthomorphism exists and is also bounded. By \cite[Theorem 2.40]{AB}, the modulus of an orthomorphism $T$ exists and satisfies $|T|(x)=|T(x)|$ for each $x\in X_{+}$. Therefore, it is easy to see that if $T$ is either $nb$-bounded or $bb$-bounded or continuous, then so is its modulus.
\end{proof}
Furthermore, we have a domination property, as well.
\begin{proposition}
Suppose $X$ is a locally solid vector lattice and $T,S$ are linear operators on $X$ such that $0\leq T\leq S$. Then we have the following observations.
\begin{itemize}
\item[\em (i)] {If $S\in Orth_{n}(X)$ then  $T\in Orth_{n}(X)$}.
\item[\em (ii)]{If $S\in Orth_{b}(X)$ then  $T\in Orth_{b}(X)$}.
\item[\em (iii)] {If $S\in Orth_{c}(X)$ then  $T\in Orth_{n}(X)$}.
\end{itemize}
\end{proposition}
\begin{proof}
$(i)$. There exists a solid zero neighborhood $U\subseteq X$ such that $S(U)$ is bounded; moreover $S$ is order bounded and band preserving. We need show that $T$ is also order bounded, band preserving and $nb$-bounded. Choose arbitrary solid zero neighborhood $V\subseteq X$. Find scalar $\gamma >0$ with $S(U)\subseteq \gamma V$. For each $x\in U_{+}$, we have $0\leq T(x)\leq S(x)\in \gamma V$ so that $T(x)\in V$ since $V$ is solid. Since $U\subseteq U_{+}-U_{-}$, we conclude that $T(U)$ is also bounded. It is clear that $T$ is also order bounded. Now, suppose $x,y\in X$ such that $x\bot y$. By \cite[Theorem 2.36 and Theorem 2.40]{AB}, we have
\[|T(x)|\wedge |y|=T(|x|)\wedge |y|\leq S(|x|)\wedge |y|=|S(x)|\wedge |y|=0,\]
so that $T(x)\bot y$.

$(ii)$. It is similar to the part $(i)$; just if necessary, replace a bounded set $B$ with its solid hull which is also bounded and use the inclusion $B\subseteq B_{+}-B_{-}$.

$(iii)$. Similar to the part $(i)$, we conclude that $T$ is a orthomorphism. Choose arbitrary solid zero neighborhood $W\subseteq X$. Find solid zero neighborhood $V\subseteq X$ with $V-V\subseteq W$. There exists a solid zero neighborhood $U\subseteq X$ with $S(U)\subseteq V$ so that $S(x)\in V$ for each positive $x\in U$. This implies that $T(x)\in V$, as well; since $V$ is solid. Therefore, $T(U_{+})\subseteq V$. We conclude that $T(U)\subseteq T(U_{+})-T(U_{-})\subseteq V-V\subseteq W$, as claimed.
\end{proof}
\begin{proposition}
Suppose $X$ is a locally solid vector lattice. Then $Orth_{n}(X)$, $Orth_{b}(X)$ and $Orth_{c}(X)$ are locally solid vector lattices.
\end{proposition}
\begin{proof}
For $Orth_{n}(X)$: by Lemma \ref{10}, $Orth_{n}(X)$ is an Archimedean vector lattice. By \cite[Theorem 2.17]{AB1}, it suffices to prove that the lattice operations in $Orth_{n}(X)$ are uniformly continuous. Suppose $(T_{\alpha})$ is a net of $nb$-bounded orthomorphisms on $X$ which converges uniformly on some zero neighborhood to zero. It is enough to show that $|T_{\alpha}|\rightarrow 0$. There exists a zero neighborhood $U\subseteq X$ such that for each zero neighborhood $V\subseteq X$ there is an $\alpha_0$ such that $T_{\alpha}(U)\subseteq V$ for each $\alpha\geq \alpha_0$.
So, for each $x\in U_{+}$, by \cite[Theorem 2.40]{AB}, $|T_{\alpha}|(x)=|T_{\alpha}(x)|\in V$ for sufficiently large $\alpha$. Note that $U$ and $V$ are solid so that the proof would be complete.

The proof for $Orth_{b}(X)$ is similar to the case of $Orth_{n}(X)$; just, we may assume that every bounded set $B\subseteq X$ is solid otherwise, consider the solid hull of $B$ which is bounded, certainly.

For $Orth_{c}(X)$: by Lemma \ref{10}, $Orth_{c}(X)$ is an Archimedean vector lattice. By \cite[Theorem 2.17]{AB1}, it suffices to prove that the lattice operations in $Orth_{c}(X)$ are uniformly continuous. Suppose $(T_{\alpha})$ is a net of continuous orthomorphisms on $X$ which converges equicontinuously to zero. It is enough to show that $|T_{\alpha}|\rightarrow 0$. For every arbitrary solid zero neighborhood $V\subseteq X$, there exists a solid zero neighborhood $U\subseteq X$ such that for each $\varepsilon>0$, there is an $\alpha_0$ with $T_{\alpha}(U)\subseteq \varepsilon V$ for each $\alpha\geq \alpha_0$.
So, for each $x\in U_{+}$, by \cite[Theorem 2.40]{AB}, $|T_{\alpha}|(x)=|T_{\alpha}(x)|\in \varepsilon V$ for sufficiently large $\alpha$. Note that $U$ and $V$ are solid so that the proof would be complete.
\end{proof}
\begin{theorem}\label{1}
Suppose $X$ is a locally solid vector lattice. If $X$ possesses the $AM$-property, then so is $Orth_{n}(X)$, $Orth_{b}(X)$ and $Orth_{c}(X)$.
\end{theorem}
\begin{proof}
For $Orth_{n}(X)$: assume that $X$ possesses the $AM$-property and $D\subseteq Orth_{n}(X)$ is bounded. This means that $D$ is bounded uniformly on some zero neighborhood. So, there is a zero neighborhood $U\subseteq X$ such that for arbitrary zero neighborhood $V\subseteq X$, there exists a positive scalar $\alpha$ such that $\{T(U): T\in D\}\subseteq \alpha V$. In particular, for any $x\in U$, the set $B_{x}=\{T(x): T\in D\}\subseteq \alpha V$. This implies that, by the assumption, ${B_{x}}^{\vee}$ is also bounded in $X$. Again, by \cite[Theorem 2.43]{AB} and using \cite[Lemma 3]{Z2}, $(T_1\vee\ldots\vee T_n)(x)=T_1(x)\vee\ldots\vee T_n(x)\in \alpha V\vee\ldots\vee \alpha V=\alpha V$ for each $T_1,\ldots,T_n\in D$. So, $(T_1\vee\ldots\vee T_n)(U)\subseteq \alpha V$. This would complete the proof.

The proof for $Orth_{b}(X)$ is similar to the case of $Orth_{n}(X)$; just, we may assume that every bounded set $B\subseteq X$ is solid otherwise, consider the solid hull of $B$ which is bounded, certainly.

For $Orth_{c}(X)$: assume that $X$ possesses the $AM$-property and $D\subseteq Orth_{c}(X)$ is bounded. This means that $D$ is bounded equicontinuously. Therefore, for arbitrary zero neighborhood $V\subseteq X$, there exists a zero neighborhood $U\subseteq X$ with $T(U)\subseteq V$ for each $T\in D$.  In particular, for any $x\in U$, the set $B_{x}=\{T(x): T\in D\}\subseteq  V$. This implies that, by the assumption, ${B_{x}}^{\vee}$ is also bounded in $X$. Again, by \cite[Theorem 2.43]{AB} and using \cite[Lemma 3]{Z2}, $(T_1\vee\ldots\vee T_n)(x)=T_1(x)\vee\ldots\vee T_n(x)\in V\vee\ldots\vee V= V$ for each $T_1,\ldots,T_n\in D$. So, $(T_1\vee\ldots\vee T_n)(U)\subseteq  V$. This would complete the proof.
\end{proof}
\begin{remark}
Observe that the converse of Theorem \ref{1} is not true, in general. Consider $X=\ell_1$; it does not have the $AM$-property. By \cite[Theorem 4.75 and Theorem 4.77]{AB} , $Orth_n(X)$ is an $AM$-space with unit so that a $C(K)$-space for some compact Hausdorff space $K$. Therefore, $Orth_{n}(X)$ possesses the $AM$-property.
\end{remark}
Nevertheless, when we consider a locally solid $f$-algebra $X$, we can have the converse, as well. Before this, we have two useful facts. Recall that by an $f$-algebra $X$, we mean a Riesz algebra such that given $x,y\in X$ with $x\wedge y=0$, we have $zx\wedge y=xz\wedge y=0$ for each positive $z\in X$; for a comprehensive context, see \cite[Section 2.3]{AB}.
\begin{theorem}\label{4}
Suppose $X$ is a locally solid $f$-algebra with a multiplication unit $e$. Then there is an $f$-algebra isomorphism homeomorphism from $X$ onto $Orth_{b}(X)$.
\end{theorem}
\begin{proof}
By \cite[Theorem 2.62]{AB}, there is an $f$-algebra isomorphism from $X$ onto $Orth(X)$ defined by $u\rightarrow T_u$ such that $T_u(x)=ux$. Now, consider this mapping from $X$ into $Orth_{b}(X)$; note that each $T_u$ is $bb$-bounded using \cite[Proposition 2.1]{MZ}. Furthermore, another using of \cite[Theorem 2.62]{AB}, convinces that this mapping is also onto and an $f$-isomorphism. So, it is enough to show that it is a homeomorphism. Suppose $(u_{\alpha})$ is a null net in $X$. We need show that $T_{u_{\alpha}}$ is null in the topology of uniform convergence on bounded sets. Fix a bounded set $B\subseteq X$ and choose arbitrary zero neighborhood $V\subseteq X$. By another application of \cite[Proposition 2.1]{MZ}, we find a zero neighborhood $U\subseteq X$ such that $UB\subseteq V$. There exists an $\alpha_0$ with $u_{\alpha}\in U$ for each $\alpha\geq\alpha_0$ so that $T_{u_{\alpha}}(B)=u_{\alpha}B\subseteq V$ for sufficiently large $\alpha$. For the converse, suppose $(T_{u_{\alpha}})$ is a null net in $Orth_{b}(X)$. This means that for each $x\in X$, $T_{u_{\alpha}}(x)\rightarrow 0$ in $X$. Put $x=e$, the multiplication unit of $X$. We see that $u_{\alpha}\rightarrow 0$, as claimed.
\end{proof}
\begin{theorem}
Suppose $X$ is a locally solid $f$-algebra with a multiplication unit $e$. Then there is an $f$-algebra isomorphism homeomorphism from $X$ onto $Orth_{c}(X)$.
\end{theorem}
\begin{proof}
The proof has the same line as in the proof of Theorem \ref{4}. By \cite[Theorem 2.62]{AB}, there is an $f$-algebra isomorphism from $X$ onto $Orth(X)$ defined by $u\rightarrow T_u$ such that $T_u(x)=ux$. Now, consider this mapping from $X$ into $Orth_{c}(X)$; note that each $T_u$ is continuous using \cite[Proposition 2.1]{MZ}. Furthermore, another using of \cite[Theorem 2.62]{AB}, convinces that this mapping is also onto and an $f$-isomorphism. So, it is enough to show that it is a homeomorphism. Suppose $(u_{\alpha})$ is a null net in $X$. We need show that $T_{u_{\alpha}}$ is null equicontinuously. Choose arbitrary zero neighborhood $V\subseteq X$. there exists a zero neighborhood $U\subseteq X$ such that $UU\subseteq V$. For each $\varepsilon>0$,  there exists an $\alpha_0$ with $u_{\alpha}\in \varepsilon U$ for each $\alpha\geq\alpha_0$ so that $T_{u_{\alpha}}(U)=u_{\alpha}U\subseteq \varepsilon UU\subseteq \varepsilon V$ for sufficiently large $\alpha$. For the converse, suppose $(T_{u_{\alpha}})$ is a null net in $Orth_{c}(X)$. This means that for each $x\in X$, $T_{u_{\alpha}}(x)\rightarrow 0$ in $X$. Put $x=e$, the multiplication unit of $X$. We see that $u_{\alpha}\rightarrow 0$, as claimed.
\end{proof}
\begin{corollary}\label{2}
Suppose $X$ is a locally solid $f$-algebra with a multiplication unit $e$. If $Orth_b(X)$ possesses the $AM$-property, then so is $X$.
\end{corollary}

\begin{corollary}
Suppose $X$ is a locally solid $f$-algebra with a multiplication unit $e$. If $Orth_c(X)$ possesses the $AM$-property, then so is $X$.
\end{corollary}

\begin{theorem}\label{3}
Suppose $X$ is a locally solid vector lattice. If $X$ possesses the Levi property, then so are $Orth_{n}(X)$, $Orth_{b}(X)$, and $Orth_{c}(X)$.
\end{theorem}
\begin{proof}
We prove the result for $Orth_{n}(X)$; the proofs for other cases are similar. Assume that $(T_{\alpha})$ is a bounded increasing net in ${Orth_{n}(X)}_{+}$. The general idea follows from the proof of \cite[Theorem 2.15]{Z1}. This implies that there is a zero neighborhood $U\subseteq X$ such that $(T_{\alpha}(U))$ is bounded for each $\alpha$. So, for each $x\in X_{+}$, the net $(T_{\alpha}(x))$ is bounded and increasing in $X_{+}$ so that it has a supremum, namely, $\alpha_x$. Define $T:X_{+}\to X_{+}$ via $T(x)=\alpha_x$. It is an additive map; it is easy to see that $\alpha_{x+y}\leq \alpha_x+\alpha_y$. For the converse, fix any $\alpha_0$. For each $\alpha\geq\alpha_0$, we have $T_{\alpha}(x)\leq \alpha_{x+y}-T_{\alpha}(y)\leq \alpha_{x+y}-T_{\alpha_0}(y)$ so that $\alpha_x\leq \alpha_{x+y}-T_{\alpha_0}(y)$. Since $\alpha_0$ was arbitrary, we conclude that $\alpha_{x}+\alpha_{y}\leq\alpha_{x+y}$. By \cite[Theorem 1.10]{AB}, it extends to a positive operator $T:X\to X$. It is clear that $T$ is also order bounded. We need show that it is band preserving. By \cite[Theorem 2.36]{AB}, it is sufficient to prove that $x\bot y$ implies that $Tx\bot y$. Note that each $T_{\alpha}$ is band preserving so that
\[|T(x)|\wedge |y|= (\bigvee_{\alpha}|T_{\alpha}(x)|)\wedge |y|=\bigvee_{\alpha}(|T_{\alpha}(x)|\wedge |y|)=0.\]
Suppose $V$ is an arbitrary order closed zero neighborhood in $X$. There is a positive scalar $\gamma$ with $T_{\alpha}(U)\subseteq \gamma V$. This means that $T(U)\subseteq \gamma V$ since $V$ is order closed.
\end{proof}
\begin{remark}
Observe that the converse of Theorem \ref{3} is not true, in general. Consider the Banach lattice $E$ consists of all piecewise linear continuous functions on $[0,1]$. By \cite[Theorem 6]{Za}, $Orth_{n}(E)=Orth(E)$ is the space $\{\alpha I: \alpha\in \Bbb R\}$, where $I$ denotes the identity operator on $E$. Clearly $Orth_{n}(E)$ possesses the Levi property but $E$ fails to have the Levi property ( it is not even order complete). Moreover, this example also shows that the Lebesgue property does not transfer from the space of all bounded orthomorphisms into the underlaying space, as well.
\end{remark}
Nevertheless, when we deal with the locally solid $f$-algebras with unit, compatible with Theorem \ref{4}, we have good news.
\begin{corollary}
Suppose $X$ is a locally solid $f$-algebra with a multiplication unit $e$. If either $Orth_b(X)$ or $Orth_{c}(X)$ possesses the Levi property, then so is $X$.
\end{corollary}

\begin{corollary}
Suppose $X$ is a locally solid $f$-algebra with a multiplication unit $e$. If either $Orth_b(X)$ or $Orth_{c}(X)$ possesses the Lebesgue property, then so is $X$.
\end{corollary}
Recall that a linear operator $T$ between vector lattices $X$ and $Y$ is said to be disjoint preserving provided that for each $x,y\in X$ with $x\bot y$, we have $T(x)\bot T(y)$. For more information, see \cite{AB}.
\begin{lemma}
Suppose $X,Y$ are locally solid vector lattices such that $Y$ is Hausdorff and $(T_{\alpha})$ is a net of disjoint preserving operators from $X$ into $Y$ which is convergent pointwise to the operator $T$. Then $T$ is also disjoint preserving.
\end{lemma}
\begin{proof}
Choose arbitrary solid zero neighborhood $W\subseteq Y$; there exists a zero neighborhood $V\subseteq Y$ with $V+V\subseteq W$. Consider $x,y\in X$ with $x\bot y$. Find an index $\alpha_0$ with $(T_{\alpha_0}-T)(x)\in V$ and $(T_{\alpha_0}-T)(y)\in V$. We claim that $T(x)\bot T(y)$. By the Birkhoff's inequality, we have
\[0\leq||T(x)|\wedge |T(y)|-|T_{\alpha_0}(x)|\wedge|T_{\alpha_0}(y)||\]
\[=||T(x)|\wedge |T(y)|-|T_{\alpha_0}(x)|\wedge |T(y)|-|T_{\alpha_0}(x)\wedge |T(y)|-T_{\alpha_0}(x)|\wedge|T_{\alpha_0}(y)||\]
\[\leq |T_{\alpha_0}(x)- T(x)|+|T_{\alpha_0}(y)- T(y)|\in V+V=W.\]
Since $T_{\alpha_0}$ is disjoint preserving, $|T_{\alpha_0}(x)|\wedge|T_{\alpha_0}(y)|=0$. Since $W$ is solid, we conclude that $|T(x)|\wedge |T(y)|\in W$. This happens for each arbitrary solid zero neighborhood $W\subseteq Y$. Therefore, $|T(x)|\wedge |T(y)|=0$ as claimed.
\end{proof}
\begin{corollary}\label{7}
Suppose $X$ is a Hausdorff locally solid vector lattice and $(T_{\alpha})$ is a net of band preserving operators on $X$ which is convergent pointwise to the operator $T$. Then $T$ is also band preserving.
\end{corollary}
Observe that, in general, the uniform limit of a sequence of order bounded operators need not be order bounded ( see \cite[Example 5.6]{AB}, due to Krengle); note that, in the example, each $K_n$ is order bounded and $K_n\rightarrow T$ uniformly, nevertheless, we can choose the coefficients $(\alpha_n)$ such that $T$ is not even order bounded ( the modulus does not exist). Furthermore, $K_n\uparrow T$, but $T$ is not order bounded, as mentioned. Consider this point that in the example, the underlying space $E$, is not an $AM$-space. Now, we focus on locally solid vector lattices. Compatible with Corollary \ref{7}, \cite[Corollary 2.7]{Z1} and \cite[Lemma 3.1 and Lemma 3.2]{HMZ}, we have the following.
\begin{corollary}
Suppose $X$ is a Hausdorff locally solid vector lattice with the Levi and $AM$-properties. Then, $Orth_{b}(X)$ and $Orth_{c}(X)$, with respect to the assumed topologies, are topologically complete. Moreover, they form bands in $Orth(X)$.
\end{corollary}
\begin{remark}
Note that we are able to consider some results presented for $Orth_{b}(X)$ and $Orth_{c}(X)$, for $Orth_{n}(X)$, as well. Just, observe that $Orth_{n}(X)$ is not topologically complete as shown in \cite[Example 2.22]{Tr}.
\end{remark}
\end{document}